\documentclass[12pt]{amsart}
\usepackage{amsfonts, amsbsy, amsmath, amssymb, mathtools}
\usepackage{color,enumerate}
\newtheorem{thm}{Theorem}[section]

\newtheorem{rmk}[thm]{Remark}

\newtheorem{thm-con}[thm]{Theorem-Conjecture}
\numberwithin{equation}{section}

\theoremstyle{definition}

\newcommand{\f}{\Bbb F}

\begin{document}

\title[DO polynomials and reversed Dickson polynomials]{Dembowski-Ostram polynomials and reversed Dickson polynomials}

\author[Neranga Fernando]{Neranga Fernando}
\address{Department of Mathematics,
Northeastern University, Boston, MA 02115, USA}
\email{w.fernando@northeastern.edu}

\begin{abstract}
We give a complete classification of Dembowski-Ostram polynomials from reversed Dickson polynomials in odd characteristic. 
\end{abstract}

\keywords{Finite field, reversed Dickson polynomial, Dembowski-Ostram polynomial}

\subjclass[2010]{11T55, 05A10, 11T06}

\maketitle

%%%%%%%%%%%%%%%%%%%%%%%%%%%%%%%%%%%%%%%%
%  section 1
%%%%%%%%%%%%%%%%%%%%%%%%%%%%%%%%%%%%%%%%

\section{Introduction}

Let $p$ be a prime, $e$ a positive integer, and $q=p^e$. Let $\Bbb F_{q}$ be the finite field with $q$ elements. Dembowski-Ostram (DO) polynomials over a finite field $\f_q$ are those of the form $\sum_{i,j}\,a_{ij}x^{p^i+p^j}$, where $a_{ij}\in \f_q$.  Dembowski-Ostram polynomials are used for a cryptographic application in the public key cryptosystem HFE (\cite{Patarin-1996}). This class of polynomials were introduced by Dembowski and Ostram in \cite{DO-1968} for constructions of planar functions in odd characteristic. A polynomial $g\in \f_q[x]$ is called a {\it planar polynomial} if $g(x+a)-g(x)$ is a permutation polynomial for every $a\in \f_q^*$. A polynomial $f \in \Bbb F_{q}[x]$ is called a \textit{permutation polynomial} (PP) over $\Bbb F_{q}$ if the associated mapping $x\mapsto f(x)$ is a bijection from $\f_{q}$ to $\f_{q}$. Clearly, a polynomial cannot be planar in even characteristic. 

In the study of permutation polynomials over finite fields, Dickson polynomials have played a pivotal role. 

The $n$-th Dickson polynomial of the first kind $D_n(x,a)$ is defined by
\[
D_{n}(x,a) = \sum_{i=0}^{\lfloor\frac n2\rfloor}\frac{n}{n-i}\dbinom{n-i}{i}(-a)^{i}x^{n-2i},
\]
where $a\in \f_q$ is a parameter. 

The $n$-th Dickson polynomial of the second kind $E_n(x,a)$ is defined by
\[
E_{n}(x,a) = \sum_{i=0}^{\lfloor\frac n2\rfloor}\dbinom{n-i}{i}(-a)^{i}x^{n-2i},
\]
where $a\in \f_q$ is a parameter. 

DO polynomials from Dickson polynomials of the first kind and second kind were completely classified by Coulter and Matthews in \cite{Coulter-Matthews-2010}. 

The concept of the reversed Dickson polynomial $D_{n}(a,x)$ was first introduced by Hou, Mullen, Sellers and Yucas in \cite { Hou-Mullen-Sellers-Yucas-FFA-2009} by reversing the roles of the variable and the parameter in the Dickson polynomial $D_{n}(x,a)$. 

The $n$-th reversed Dickson polynomial of the first kind $D_n(a,x)$ is defined by

\begin{equation}\label{E1.2}
D_{n}(a,x) = \sum_{i=0}^{\lfloor\frac n2\rfloor}\frac{n}{n-i}\dbinom{n-i}{i}(-x)^{i}a^{n-2i},
\end{equation}

where $a\in \f_q$ is a parameter. 

By reversing the roles of the variable and the parameter in the Dickson polynomial of the second kind $E_{n}(x,a)$, the $n$-th reversed Dickson polynomial of the second kind $E_n(a,x)$ can be defined by

\begin{equation}\label{E1.4}
E_{n}(a,x) = \sum_{i=0}^{\lfloor\frac n2\rfloor}\dbinom{n-i}{i}(-x)^{i}a^{n-2i},
\end{equation}

where $a\in \f_q$ is a parameter.

In a recent paper, X. Zhang, B. Wu and Z. Liu studied DO polynomials from reversed Dickson polynomials in characteristic 2; see \cite{Zhang-Wu-Liu-2016}. Motivated by the work of Coulter and Matthews in \cite{Coulter-Matthews-2010} and the work of X. Zhang, B. Wu and Z. Liu, we study and completely classify DO polynomials from reversed Dickson polynomials in odd characteristic. 

It is easy to see from the definitions of reversed Dickson polynomials of the first and second kinds that 
$$D_{n}(a,x)=a^n\,D_{n}(1,\frac{x}{a^2})$$

and 

$$E_{n}(a,x)=a^n\,E_{n}(1,\frac{x}{a^2}),$$

which implies that 

\begin{center}
$D_{n}(a,x)$ is DO if and only if $D_{n}(1,x)$ is DO
\end{center}

and 

\begin{center}
$E_{n}(a,x)$ is DO if and only if $E_{n}(1,x)$ is DO,
\end{center}

respectively. Since $D_{n}(1,0)=1$, $E_{n}(1,0)=1$ and DO polynomials do not contain any constant terms, we study DO polynomials arising from $D_n(1,x)- D_n(1,0)$ and $E_n(1,x)- E_n(1,0)$. The paper is organized as follows. 

We present a complete classification of DO polynomials from reversed Dickson polynomials of the first kind and second kind over in odd characteristic in Section 2 and Section 3, respectively. We also present the monomials, binomials, trinomials, and quadrinomials when reversed Dickson polynomials are DO polynomials.

Throughout the paper, we always assume that $p$ is odd.

\vskip 0.1in

\noindent \textbf{Acknowledgements}

\vskip 0.1in

The author would like to thank Ariane Masuda for the invaluable discussions while the manuscript was in preparation. She also contributed to finding patterns of parameters $n$ and $d$, and the polynomials listed in Remark~\ref{R1} and Remark~\ref{R2}. 

%%%%%%%%%%%%%%%%%%%%%%%%%%%%%%%%%%
%   section 2
%%%%%%%%%%%%%%%%%%%%%%%%%%%%%%%%%%

\section{DO polynomials from reversed Dickson polynomials of the first kind}

We first consider reversed Dickson polynomials of the first kind. Recall that 
\begin{center}
$D_{n}(a,x)$ is DO if and only if $D_{n}(1,x)$ is DO.
\end{center}

Let $d$ be a positive integer. We denote $D_n(1,x^d)- D_n(1,0)$ by $\mathcal{D}_n$. Then

\begin{equation*}
\mathcal{D}_n= \sum_{i=1}^{\lfloor\frac n2\rfloor}\frac{n}{n-i}\dbinom{n-i}{i}(-x^d)^{i}.
\end{equation*}

Since $\mathcal{D}_{np}=\mathcal{D}_n^p$ and $\mathcal{D}_n(x^{pd})=\mathcal{D}_n(x^d)^p$, we always assume that $\textnormal{gcd}(n,p)=1$ and $\textnormal{gcd}(d,p)=1$. 

\begin{thm}
Let $q$ be a power of a prime $p$. The polynomial $\mathcal{D}_n$ is a Dembowski-Ostrom polynomial over $\mathbb F_q$ if and only if one of the following holds.
\begin{enumerate}[{\normalfont (i)}]
\item $p=3$, $d=(p^j+1)p^\ell$, $n=2p^m$, where $j, m, \ell\geq 0$.
\item $p=3$, $d=2p^\ell$, $n=4p^{m}$, where $m,\ell\geq 0$.
\item $p=3$, $d=2p^\ell$, $n=5p^{m}$, where $m,\ell\geq 0$.
\item $p=3$, $d=2p^\ell$, $n=7p^{m}$, where $m,\ell\geq 0$.
\item $p>3$, $d=(p^i+1)p^{\ell}$, $n=2p^m$, where $\ell, m\geq 0$. 
\item $p>3$, $d=(p^i+1)p^{\ell}$, $n=3p^m$, where $\ell, m\geq 0$. 
\end{enumerate}
\end{thm}

\begin{proof}

We first consider the case where $n$ is odd. 

\textbf{Case 1.} $n$ is odd. 

\[ 
\begin{split}
\mathcal{D}_n&= \sum_{i=1}^{\lfloor\frac{(n-1)}{2}\rfloor}\frac{n}{n-i}\dbinom{n-i}{i}(-x^d)^{i}
\end{split}
\]

Then

\begin{equation}
\begin{split}
&\mathcal{D}_n=\cr
&-nx^d+\frac{n}{2}(n-3)x^{2d}-\frac{n}{6}(n-4)(n-5)\,x^{3d}+\cdots +(-1)^{\frac{n-3}{2}}\,\frac{(n-1)n(n+1)}{24}\,x^{\frac{d(n-3)}{2}}\cr
&+(-1)^{\frac{n-1}{2}}\,n\,x^{\frac{d(n-1)}{2}}.
\end{split}
\end{equation}

Note that the coefficient of the second term $\frac{n}{2}(n-3)$ is not always zero. 

\textbf{Subcase 1.1.} Let's assume that $\frac{n}{2}(n-3)$ is not divisible by $p$. If $\mathcal{D}_n$ is DO, $2d=p^\alpha +p^\beta$. Since $\textnormal{gcd}(d,p)=1$, $d=p^i+1$. Then 
$$2(p^i+1)=p^\alpha+p^\beta.$$

Since $p$ is odd and $p\not | \,\,(p^i+1)$, one of $\alpha$ or $\beta$ is zero, say $\beta=0$. Then
$$2(p^i+1)=p^\alpha+1$$

implies 

$$p^\alpha -2p^i=1,$$

which is true if and only if $p=3$, $\alpha =1$ and $i=0$, i.e. $d=2$. 

For the rest of the subcase we assume that $p=3$ and we prove the following. 

\begin{center}
$\mathcal{D}_n$ is DO if and only if $n=5\cdot 3^k$ or $n=7\cdot 3^k$. 
\end{center}

It is easy to see that when $n=5$ and $n=7$ we have

$$\mathcal{D}_5=x^2+2x^4$$

and

$$\mathcal{D}_7=2x^2+2x^4+2x^6,$$

respectively. Clearly they are DO polynomials. 

Now we claim that when $p=3$, $n>7$ odd and $n$ is not a multiple of a power of 3, $\mathcal{D}_n$ is never DO. 

Since $n>7$ is odd and $\textnormal{gcd}(n,3)=1$, we have either $n\equiv 2\pmod{3}$ or $n\equiv 1\pmod{3}$.

Let $n\equiv 2\pmod{3}$ and consider the term before the last term

$$(-1)^{\frac{n-3}{2}}\,\frac{(n-1)n(n+1)}{24}\,x^{\frac{d(n-3}{2}}.$$

Since $(n-1)n(n+1)=6\ell$ for some integer $\ell$, we have 
$$(-1)^{\frac{n-3}{2}}\,\frac{(n-1)n(n+1)}{24}=\frac{\ell}{4}\equiv \ell \pmod{3}.$$

If $\ell \not\equiv 0\pmod{3}$, then we claim that $\frac{d(n-3)}{2}$ is not a sum of powers of 3. Since $d=2$, $\frac{d(n-3)}{2}=n-3$. Assume to the contrary $n-3=3^i+3^j$ which implies $n-2=3^i+3^j+1$. Since $n\equiv 2\pmod{3}$, $n-2=3^i+3^j+1$ if and only if $i=j=0$. $i=j=0$ implies $n-3=2$ which is a contradiction because this is not the first term. 

Now assume that $\ell \equiv 0\pmod{3}$. In this case we show that the fourth term always exists. Note that the fourth term is $x^8$ whose exponent is not a sum of powers of 3. The coefficient of the fourth term is 
$$\frac{n}{n-4}\,\binom{n-4}{4}=\frac{n(n-5)(n-6)(n-7)}{24}.$$

Clearly $3\not | n$ and $3\not | (n-6)$. Now we show that $(n-5)(n-7)$ is a multiple of 24. 

Recall that $n>7$ is odd, $n\equiv 2\pmod{3}$ and $\textnormal{gcd}(n,3)=1$. 

Let $n=2\ell+1$, where $\ell$ is an integer. Then $(n-5)(n-7)=4(\ell-2)(\ell-3)$. Since $n\equiv 2\pmod{3}$, $\ell\equiv 2\pmod{3}$. Let $\ell = 2+3k$ for some integer $k$. Then this implies $(n-5)(n-7)=12k(3k-1)$. Notice that $k$ and $(3k-1)$ have different parity. So $k(3k-1)$ is even which says that $(n-5)(n-7)$ is a multiple of 24. 

Let  

\begin{equation}\label{Nera1}
(n-5)(n-7)=24m,
\end{equation}

where $m$ is an integer. Then 

$$\frac{n(n-5)(n-6)(n-7)}{24}=n(n-6)m.$$

Clearly, $3\not |\, n$ and $3\not |\,(n-6)$. Now we show that $3\not |\, m$. Assume to the ccontrary $3\,|\,m$. 

Then from \eqref{Nera1} we have, $(n-5)(n-7)=72e$ for some integer $e$. 

Since $n\equiv 2\pmod{3}$, write $n=3\ell_1-1$ for some integer $\ell_1$. Recall that in this case the term before the last term is zero. So

\[
\begin{split}
\frac{(n-1)n(n+1)}{24}=\frac{(3\ell_2-2)(3\ell_2-1)(3\ell_2)}{24}\equiv \ell_2 \equiv 0\pmod{3}.
\end{split}
\]

Let $\ell_2=3k_2$ for some integer $k_2$. 

Then $n=3\ell_2-1=9k_2-1\equiv -1\pmod{9}$.

From $(n-5)(n-7)=72e$ and $n\equiv -1\pmod{9}$ we have $2\equiv 0\pmod{9}$, which is a contradiction. 

Now Let $n\equiv 1\pmod{3}$.

We first look at the fourth term. Note that the fourth term is $x^8$ whose exponent is not a sum of two powers of 3. The coefficient of the fourth term is 

$$\frac{n}{n-4}\,\binom{n-4}{4}=\frac{n(n-5)(n-6)(n-7)}{24}.$$

Clearly $(n-5)(n-6)(n-7)=6\ell$ for some integer $\ell$. 

$$\frac{n(n-5)(n-6)(n-7)}{24}=\frac{n\ell}{4}\equiv n\ell \pmod{3}.$$ If $\ell \not\equiv 0\pmod{3}$, then clearly $\mathcal{D}_n$ is not DO. Now we consider the case $\ell \equiv 0\pmod{3}$ and claim that the 7th term is not zero. Note that the 7th term is $x^{14}$ whose exponent is not a sum of two powers of 3. The coefficient of the 7th term is 

\begin{equation}\label{Nera2}
\frac{n}{n-7}\,\binom{n-7}{7}=\frac{n(n-8)(n-9)(n-10)(n-11)(n-12)(n-13)}{7!}.
\end{equation}

We claim that $6\,|\,(n-13)$ and $3\,|\,(n-10)$.

By division algorithm we have $n-13=6q_1+r_1$, where $0\geq r_1\geq 5$. Since $n\equiv 1\pmod{3}$, $r_1\equiv 0\pmod{3}$ which implies $r_1=0\,\,\textnormal{or}\,\,3$. 

If $r_1=3$, then $n-13=3(2q_1+1)$ which is a contradiction since the left hand side is even and the right hand side is odd. So $r_1=0$ which implies 6 divides $(n-13)$. 

By division algorithm we have $n-10=6q_2+r_2$, where $0\geq r_2\geq 2$. Since $n\equiv 1\pmod{3}$, $r_2\equiv 0\pmod{3}$ which implies $r_2=0$. So 3 divides $(n-10)$. 

Now we claim that $q_1$ and $q_2$ are not divisible by 3. 

Assume that $q_1=3k_1$ for some integer $k_1$. Then $n-13=18k_1$ which implies $n=13+18k_1$. Recall that the coefficient of the fourth term given by  
$\frac{n(n-5)(n-6)(n-7)}{24}$ is zero in this case. Consider $(n-5)(n-6)(n-7)$. Straightforward computation yields that

$$(n-5)(n-6)(n-7)=(18k_1+8)(18k_1+7)(18k_1+6)=6\,\times\,A,$$

where $A$ is not divisble by 3 which is a contradiction. 

Now assume that $q_2=3k_2$ for some integer $k_2$. Then $n-10=9k_2$ which implies $n\equiv 1 \pmod{9}$. Let $n=9k+1$ for some integer $k\in \mathbb{Z}$. Then the coefficient of the fourth term is 

\[
\begin{split}
\frac{n(n-5)(n-6)(n-7)}{24}&=\frac{(9k+1)(9k-4)(9k-5)(9k-6)}{24}\cr
&=\frac{(9k+1)(9k-4)(9k-5)(3k-2)}{8}\cr
&\not\equiv 0\pmod{3},
\end{split}
\]

which is a contradiction. 

Now let's go back to the coefficient of the 7th term. From \eqref{Nera2} we have

\[
\begin{split}
\frac{n}{n-7}\,\binom{n-7}{7}&=\frac{n(n-8)(n-9)(n-10)(n-11)(n-12)(n-13)}{7!}\cr
&=\frac{n(n-8)(n-9)(n-11)(n-12)\,q_1q_2}{7\cdot 5\cdot 4\cdot 2\cdot 1}\cr
\end{split}
\]

Clearly $3\not|\,n$, $3\not|\,(n-9)$ and $3\not|\,(n-12)$. 

If $3\,|\,(n-8)$, then $n-8=3e_1$ for some integer $e_1$. This implies $2\equiv 0\pmod{3}$ which is a contradiction. 

If $3\,|\,(n-11)$, then $n-11=3e_2$ for some integer $e_2$. This implies $2\equiv 0\pmod{3}$ which is a contradiction. 

Therefore the coefficient of the 7th term is non zero. 

\textbf{Subcase 1.2.} Now let's consider the case where $p>3$. Recall 

\begin{equation*}
\mathcal{D}_n= \sum_{i=1}^{\lfloor\frac n2\rfloor}\frac{n}{n-i}\dbinom{n-i}{i}(-x^d)^{i}.
\end{equation*}

We prove the following. 

\begin{center}
$\mathcal{D}_n$ is a DO if and only if $n=3\cdot p^k$. 
\end{center}

Since $\textnormal{gcd}(d,p)=1$, $d=p^i+1$. 

When $n=3$, we have $\mathcal{D}_3=-3x^d=(p-2)x^{p^i+1}$, which is clearly DO. 

Now we claim that when $n\neq 3\cdot p^k$, $\mathcal{D}_n$ is not DO. Recall that $\textnormal{gcd}(n,p)=1$. 

Let's consider the second term $\frac{n}{2}(n-3)$. If this is not zero, from subcase 1.1 we have $p=3$, which is a contradiction. So, if the second term is not zero, the polynomial is not DO. 

If the second term is zero, i.e. $\frac{n}{2}(n-3)\equiv 0\pmod{p}$, then $n\equiv 3\pmod{p}$. In this case, we show that the last term is not DO. 

Notice that the last term is $(-1)^{\frac{n}{2}}\cdot 2\cdot x^{\frac{dn}{2}}$. The coefficient is clearly not zero. Assume to the contrary 
$$\frac{dn}{2}=p^i+p^j.$$

Since $\textnormal{gcd}(n,p)=1$ and $\textnormal{gcd}(d,p)=1$, we have 

$$\frac{dn}{2}=p^i+1,$$

which implies

$$dn=2(p^i+1).$$

Since $n\equiv 3\pmod{p}$, we have $3d\equiv 2\pmod{p}$ or $3d\equiv 4\pmod{p}$. Let $d=p^{\ell}+1$. Thus  $3(p^j+1)\equiv 2\pmod{p}$ or $3(p^j+1)\equiv 4\pmod{p}$, which is a contradiction for any nonnegative $j$. 

\textbf{Case 2.} $n$ is even. Then we have

\begin{equation}
\begin{split}
&\mathcal{D}_n=\cr
&-nx^d+\frac{n}{2}(n-3)x^{2d}-\frac{n}{6}(n-4)(n-5)\,x^{3d}+\cdots +(-1)^{\frac{n}{2}-1}\,\frac{n^2}{4}\,x^{d(\frac{n}{2}-1)}+(-1)^{\frac{n}{2}}\cdot 2\cdot x^{\frac{dn}{2}}.
\end{split}
\end{equation}

\textbf{Subcase 2.1.} Let's first consider the case where the coefficient of the second term is not zero, i.e. $n\not\equiv 3\pmod{p}$. Since $\textnormal{gcd}(p,d)=1$, we have $2d=p^i+1$ for some nonnegative integer $i$. Also, the first term is nonzero. So, $d=p^\alpha+1$ for some nonnegative integer $\alpha$. Now $2d=p^i+1$ implies $p^i-2p^{\alpha}=1$. 

\begin{center}
$p^i-2p^{\alpha}=1$ if and only if $p=3, i=1, \alpha =0$. 
\end{center}

$\alpha =0$ implies $d=2$. Therefore second term is nonzero means $p=3$ and $d=2$. Now we show that $\mathcal{D}_n$ is DO if and only if $n=4\cdot 3^k$. Since the second term is zero, $n>2$. 

Let $n=4$. Then $\mathcal{D}_4=2x^2+2x^4$, which is clearly DO.

Let $n>4$ be even and $\textnormal{gcd}(n,3)=1$. Now we claim that $\mathcal{D}_n$ is not DO. 

Consider the fourth term and its coefficient. Because, fourth term is $x^8$, which is clearly not DO. Coefficient of the fourth term is 

$$\frac{n}{n-4}\,\binom{n-4}{4}=\frac{n(n-5)(n-6)(n-7)}{24}.$$

Note that $n\neq 6$. Let $(n-5)(n-6)(n-7)=6\ell_1$ for some integer $\ell_1$. Then 

$$\frac{n}{n-4}\,\binom{n-4}{4}\equiv n\ell_1 \pmod{3}.$$

If $\ell_1 \not\equiv 0\pmod{3}$, then $\mathcal{D}_n$ is not DO. 

Now assume that $\ell_1\equiv 0\pmod{3}$. Then we claim that the last term is not DO. Consider the exponenet of the last term which is $\frac{dn}{2}$. 

Recall that the second term is nonzero, i.e. $n\not\equiv 0\pmod{3}$, which implies $3\not | n$. Since the coefficient of the last term is nonzero, $\textnormal{gcd}(3,d)=1$ and $\textnormal{gcd}(3,n)=1$, for some integer $\ell_2$ we have

$$\frac{dn}{2}=3^{\ell_2}+1.$$

Since $d=2$, the above implies $n=3^{\ell_2}+1$.

Clearly $\ell_2\neq 0$. Because $\ell_2=0$ implies $n=2$. A contradiction. 

If $\ell_2>0$, then $n\equiv 1\pmod{3}$. Now consider the term before the the last term: $(-1)^{\frac{n}{2}-1}\,\frac{n^2}{4}\,x^{d(\frac{n}{2}-1)}$. Clearly, the coefficient is nonzero since $\textnormal{gcd}(n,3)=1$. Consider its exponent $d(\frac{n}{2}-1)=2(\frac{n}{2}-1)=n-2$. 

Since $\mathcal{D}_n$ is DO, $n-2=3^{\alpha}+3^{\beta}$. 

If $\alpha>0, \beta=0$, then $n=3^{\alpha}+3$, which is a contradiction since $\textnormal{gcd}(n,3)=1$.

If $\beta=0$, then $\alpha =0$, i.e. $n=4$. This contradicts the assumption that $n>4$. 

If $\alpha>0, \beta>0$, then $n\equiv 2\pmod{3}$ which contradicts the fact that $n\equiv 1\pmod{3}$.

\textbf{Subcase 2.2.} Now we consider the case where the coefficient of the second term is zero, i.e. $n\equiv 3\pmod{p}$. 

We prove the following. 

\begin{center}
$\mathcal{D}_n$ is a DO if and only if $n=2\cdot p^k$. 
\end{center}

When $n=2$, we have $\mathcal{D}_2=-2x^d=(p-2)x^{p^i+1}$, which is clearly DO for $p\geq 3$. 

Now we claim that when $n\neq 2\cdot p^k$ and $p\geq 3$, $\mathcal{D}_n$ is not DO. Recall that $\textnormal{gcd}(n,p)=1$.

Assume that $n\neq 2\cdot p^k$ and $n\equiv 3\pmod{p}$. Since the last term is nonzero, consider $x^{\frac{dn}{2}}$ (note that when $p=3$, last term is the only term of the polynomial). Since $\textnormal{gcd}(n,p)=1$ and $\textnormal{gcd}(d,p)=1$, let $\frac{dn}{2}=p^i+1$ and $d=p^j+1$ for some nonnegative integers $i$ and $j$. Then we have 

$$(p^j+1)n=2(p^i+1).$$

If $j=0$, then $n=p^i+1$, which implies $n\equiv 2\pmod{p}$ or $n\equiv 1\pmod{p}$ depending on whether $i=0$ or $i>0$, respectively. A contradiction. 

If $j>0$, we have $p^jn+n=2p^i+2$. If $i>0$, then $n\equiv 2\pmod{p}$. A contradiction. If $i=0$, $p^jn+n=4$, which implies $n\equiv 4\pmod{p}$. A contradiction.  

This completes the proof. 

\end{proof}

\begin{rmk}\label{R1}
The DO polynomials obtained in the previous theorem are monomials, binomials or trinomials. We list them below. 
\begin{enumerate}
\item $p=3$, $d=(p^i+1)p^\ell$, $n=2p^m$, $\mathcal{D}_n=x^{p^{\ell+m+i}+p^{\ell +m}}$. 
\item $p=3$, $d=2p^\ell$, $n=4p^m$, $\mathcal{D}_n=2x^{2\cdot p^{\ell +m}}+2x^{4\cdot p^{\ell +m}}$. 
\item $p=3$, $d=2p^\ell$, $n=5p^m$, $\mathcal{D}_n=x^{2\cdot p^{\ell+m}}+2x^{4\cdot p^{\ell +m}}$. 
\item $p=3$, $d=2p^\ell$, $n=7p^m$, $\mathcal{D}_n=2x^{2\cdot p^{\ell+m}}+2x^{4\cdot p^{\ell +m}}+x^{6\cdot p^{\ell+m}}$. 
\item $p>3$, $d=(p^i+1)p^\ell$, $n=2p^m$, $\mathcal{D}_n=(p-2)\,x^{dp^m}$. 
\item $p>3$, $d=(p^i+1)p^\ell$, $n=3p^m$, $\mathcal{D}_n=(p-3)\,x^{dp^m}$. 
\end{enumerate}
\end{rmk}

%%%%%%%%%%%%%%%%%%%%%%%%%%%%%%%%%%
%   section 3
%%%%%%%%%%%%%%%%%%%%%%%%%%%%%%%%%%

\section{DO polynomials from reversed Dickson polynomials of the second kind}

In this Section, we consider reversed Dickson polynomials of the second kind. Recall that 
\begin{center}
$E_{n}(a,x)$ is DO if and only if $E_{n}(1,x)$ is DO.
\end{center}

Let $d$ be a positive integer. We denote $E_n(1,x^d)- E_n(1,0)$ by $\mathsf{E}_n$. Then

\begin{equation*}
\mathsf{E}_n= \sum_{i=1}^{\lfloor\frac n2\rfloor}\dbinom{n-i}{i}(-x^d)^{i}.
\end{equation*}

\noindent In this case, $\mathsf{E}_{np}\neq \mathsf{E}_n^p$, but $\mathsf{E}_n(x^{pd})=\mathsf{E}_n(x^d)^p$. So we always assume that $\textnormal{gcd}(d,p)=1$. 

We first consider the case $p=3$. 

\begin{thm} \label{T2.3}
Let $p=3$. The polynomial $\mathsf{E}_n$ is a Dembowski-Ostrom polynomial over $\mathbb F_q$ if and only if one of the following holds.
\begin{enumerate}
\item [(i)] $n=2, 3, 5$, or $6$, $d=(p^{\alpha}+1)p^k$. 
\item [(ii)] $n=4$, $d=\Big(\frac{p^{\alpha}+1}{2}\Big)p^k$. 
\item [(iii)] $n=7$, $d=2p^k$. 
\item [(iv)] $n=10$, $d=2p^k$. 
\item [(v)] $n=13$, $d=2p^k$.
\item [(vi)] $n=15$, $d=4p^k$.
\item [(vii)] $n=19$, $d=2p^k$.
\end{enumerate}
\end{thm}

\begin{proof}

We have $\mathsf{E}_2=2x^d$, $\mathsf{E}_3=x^d$, $\mathsf{E}_5=2x^{d}$, and $\mathsf{E}_6=x^d+2x^{3d}$. It is clear that these polynomials are DO if and only if $d=p^{\alpha}+1$.

Now consider the case $n=15$. We have $\mathsf{E}_{15}=x^d+2x^{3d}+x^{7d}$. Assume that $\mathsf{E}_{15}$ is DO. Then $d=3^{\alpha}+1$ and $7d=3^i+1$, which implies $7(3^{\alpha}+1)=3^i+1$ if and only if $\alpha = 1$, and $i=3$, i.e. $d=4$. 

Let $n=4$. $\mathsf{E}_4=x^{2d}$, which is DO if and only if $d=\frac{p^{\alpha}+1}{2}$. 

Let $n=7$. Then $\mathsf{E}_7=x^{2d}+2x^{3d}$, which is DO if and only if $d=\frac{3^{\alpha}+1}{2}$ and $3d=3^i+3^{i+1}$. This implies 
$$3\Big(\frac{3^{\alpha}+1}{2}\Big)=3^i+3^{i+1},$$

which is true if and only if $i=1$ and $\alpha =1$, i.e. $d=2$. 

Consider the cases $n=10$ and $n=19$. Since $\mathsf{E}_{10}=x^{2d}+x^{3d}+2x^{5d}$ and $\mathsf{E}_{19}=x^{2d}+x^{3d}+2x^{5d}+2x^{9d}$, a similar argument to that of $\mathsf{E}_7$ shows that $\mathsf{E}_{10}$ and $\mathsf{E}_{19}$ are DO if and only if $d=2$. 

Now let $n=13$. Then $\mathsf{E}_{13}=x^{2d}+x^{5d}+x^{6d}$. $\mathsf{E}_{13}$ is DO if and only if $2d=3^{\alpha}+1$ and $5d=3^i+1$. This implies
$$5\Big(\frac{3^{\alpha}+1}{2}\Big)=3^i+1,$$

which is true if and only if $i=2$ and $\alpha =1$, i.e. $d=2$. 

This completes the proof of the necessity part. 

Now we show that if $n\not \in \{ 2, 3, 4, 5, 6, 7, 10, 13, 15,19 \}$, then $\mathsf{E}_n$ is not DO. Let $d$ be any positive integer such that $\textnormal{gcd}(d,p)=1$. 

It is straightforward to see that when $n\in \{8, 9, 11, 12, 17, 18\}$, $\binom{n-1}{1}\not\equiv 0\pmod{3}$ and $\binom{n-4}{4}\not\equiv 0\pmod{3}$. This means the coefficients of the first polynomial term, $x^d$, and the fourth polynomial term, $x^{4d}$, are non-zero. Assume to the contrary that $\mathsf{E}_n$ is DO when $n\in \{8, 9, 11, 12, 17, 18\}$. Then $d=3^{\alpha}+1$ and $4d=3^i+1$. This implies $4(3^{\alpha}+1)=3^i+1$, which is not true for any $\alpha$ and $i$. Thus 
$\mathsf{E}_n$ is not DO when $n\in \{8, 9, 11, 12, 17, 18\}$. 

When $n=14$, we have $\mathsf{E}_{14}=2x^d+x^{6d}+2x^{7d}$. Assume to the contrary that $\mathsf{E}_{14}$ is DO. Then $d=3^{\alpha}+1$ and $6d=3^i+1$. This implies $6(3^{\alpha}+1)=3^i+1$, which is not true for any $\alpha$ and $i$. Thus $\mathsf{E}_{14}$ is not DO. 

When $n=16$, we have $\mathsf{E}_{16}=x^{2d}+2x^{3d}+x^{8d}$. Assume to the contrary that $\mathsf{E}_{16}$ is DO. Then $2d=3^{\alpha}+1$ and $8d=3^i+1$. This implies $8(\frac{3^{\alpha}+1}{2})=3^i+1$, which is not true for any $\alpha$ and $i$. Thus $\mathsf{E}_{16}$ is not DO. 

Now let $n>19$. We divide this into two cases. 

\noindent \textbf{Case 1.} $n\equiv 0, 2, 3, 8 \pmod{9}$. 

In this case, clearly $\binom{n-1}{1}\not\equiv 0, 3, 6\pmod{9}$, which implies $\binom{n-1}{1}\not\equiv 0\pmod{3}$. Also, $\binom{n-4}{4}\not\equiv 0, 3, 6\pmod{9}$, which implies $\binom{n-4}{4}\not\equiv 0\pmod{3}$. This means the coefficients of the first polynomial term, $x^d$, and the fourth polynomial term, $x^{4d}$, are non-zero. An argument similar to a previous argument shows that $\mathsf{E}_n$ is not DO. 

\noindent \textbf{Case 2.} $n\equiv 1, 4, 5, 6, 7 \pmod{9}$. 

\noindent \textbf{Sub Case 2.1} $n\equiv  5, 6 \pmod{9}$. 

First consider the case where $n\equiv 5, 6 \pmod{9}$. 

When $n\equiv 5, 6 \pmod{9}$, $\binom{n-1}{1}\equiv 4, 5\pmod{9}$, which implies $\binom{n-1}{1}\not\equiv 0\pmod{3}$.

When $n\equiv 5 \pmod{9}$, $\binom{n-6}{6}\equiv 1\pmod{9}$, which implies $\binom{n-6}{6}\not\equiv 0\pmod{3}$.

Since $\binom{n-1}{1}\not\equiv 0\pmod{3}$ and $\binom{n-6}{6}\not\equiv 0\pmod{3}$, we have $x^d$ and $x^{6d}$. Assume to the contrary that $\mathsf{E}_n$ is DO. Then $d=3^{\alpha}+1$ and $6d=3^i+1$, which implies $6(3^{\alpha}+1)=3^i+1$ if and only if $3^i-6\cdot 3^{\alpha}=5$, which is a contradiction. 

When $n\equiv 6 \pmod{9}$, $\binom{n-3}{3}\equiv 1\pmod{9}$, which implies $\binom{n-3}{3}\not\equiv 0\pmod{3}$.

Since $\binom{n-1}{1}\not\equiv 0\pmod{3}$ and $\binom{n-3}{3}\not\equiv 0\pmod{3}$, we have $x^d$ and $x^{3d}$. Assume to the contrary that $\mathsf{E}_n$ is DO. Then $d=3^{\alpha}+1$ and $3d=3^i+1$, which implies $3(3^{\alpha}+1)=3^i+1$ if and only if $3^i-3^{\alpha+1}=2$, which is a contradiction. 

\noindent \textbf{Sub Case 2.2} $n\equiv 1, 4, 7 \pmod{9}$.

When $n\equiv 1, 4, 7 \pmod{9}$, $\binom{n-2}{2}\not\equiv 1\pmod{9}$, which implies $\binom{n-2}{2}\not\equiv 0\pmod{3}$. 

Let $n\equiv 1\,\textnormal{or}\,7\, \pmod{9}$. Then $\binom{n-3}{3}\equiv 5\,\,\textnormal{or}\,\,4\,\pmod{9}$, which implies $\binom{n-3}{3}\not\equiv 0\pmod{3}$. 

Assume to the contrary that $\mathsf{E}_n$ is DO. Then $2d=3^{\alpha}+1$ and $3d=3^i+1$, which implies $3\Big(\frac{3^{\alpha}+1}{2}\Big)=3^i+1$ if and only if $2\cdot 3^i-3^{\alpha+1}=1$, which is a contradiction. 

Let $n\equiv 4 \pmod{9}$. Then $\binom{n-6}{6}\equiv 7 \pmod{9}$, which implies $\binom{n-6}{6}\not\equiv 0\pmod{3}$. 

Assume to the contrary that $\mathsf{E}_n$ is DO. Then $2d=3^{\alpha}+1$ and $6d=3^i+1$, which implies $6\Big(\frac{3^{\alpha}+1}{2}\Big)=3^i+1$ if and only if $2\cdot 3^i-6\cdot 3^{\alpha}=4$, which is a contradiction. 

\end{proof}

\begin{thm}\label{T2.4}
Let $p=5$. The polynomial $\mathsf{E}_n$ is a Dembowski-Ostrom polynomial over $\mathbb F_q$ if and only if one of the following holds.
\begin{enumerate}
\item [(i)] $n=2, 3$ and $d=p^k(p^{\alpha}+1)$. 
\item [(ii)] $n=7$ and $d=2p^k$. 
\end{enumerate}
\end{thm}

\begin{proof}
Let $n=2$. Then $\mathsf{E}_2=4x^d$. If $\mathsf{E}_2$ is DO, then clearly $d=p^{\alpha}+1$.

Let $n=3$. Then $\mathsf{E}_2=3x^d$. If $\mathsf{E}_3$ is DO, then clearly $d=p^{\alpha}+1$.

Let $n=7$. Then $\mathsf{E}_2=4x^d+x^{3d}$. If $\mathsf{E}_3$ is DO, then $d=5^{\alpha}+1$ and $3d=5^i+1$, which implies $5^i-3\cdot 5^\alpha =2$. This is true when $i=1$ and $\alpha =0$, i.e. $d=2$. 

Now we show that when $n\not = 2, 3, 7$, $\mathsf{E}_n$ is not DO.

Assume that $n\not = 2, 3, 7$. We divide this into four cases. 

When $n\equiv 1 \pmod{5}$, $\binom{n-2}{2}\equiv 1\pmod{5}$ and $\binom{n-3}{3}\equiv 1\pmod{5}$.

Since $\binom{n-2}{2}\not\equiv 0\pmod{5}$ and $\binom{n-3}{3}\not\equiv 0\pmod{5}$, we have $x^{2d}$ and $x^{3d}$. Assume to the contrary that $\mathsf{E}_n$ is DO. Then $2d=3^{\alpha}+1$ and $3d=3^i+1$, which implies $3\Big(\frac{3^{\alpha}+1}{2}\Big)=3^i+1$ if and only if $2\cdot 5^i-3\cdot 5^{\alpha}=1$, which is a contradiction. 

When $n\equiv 2 \pmod{5}$, $\binom{n-1}{1}\equiv 1\pmod{5}$ and $\binom{n-5}{5}\equiv 4\pmod{5}$.

Since $\binom{n-1}{1}\not\equiv 0\pmod{5}$ and $\binom{n-5}{5}\not\equiv 0\pmod{5}$, we have $x^{d}$ and $x^{5d}$. Assume to the contrary that $\mathsf{E}_n$ is DO. Then $d=3^{\alpha}+1$ and $5d=3^i+1$, which implies $5(3^{\alpha}+1)=3^i+1$ if and only if $5^i-5^{\alpha+1}=4$, which is a contradiction. 

When $n\equiv 3 \pmod{5}$, $\binom{n-1}{1}\equiv 2\pmod{5}$ and $\binom{n-4}{4}\equiv 1\pmod{5}$.

Since $\binom{n-1}{1}\not\equiv 0\pmod{5}$ and $\binom{n-4}{4}\not\equiv 0\pmod{5}$, we have $x^{d}$ and $x^{4d}$. Assume to the contrary that $\mathsf{E}_n$ is DO. Then $d=3^{\alpha}+1$ and $4d=3^i+1$, which implies $4(3^{\alpha}+1)=3^i+1$ if and only if $5^i-4\cdot 5^{\alpha}=3$, which is a contradiction. 

Now consider the two cases $n\equiv 0 \pmod{5}$ and  $n\equiv 4 \pmod{5}$.

When $n\equiv 0 \pmod{5}$, $\binom{n-1}{1}\equiv 4\pmod{5}$ and $\binom{n-2}{2}\equiv 3\pmod{5}$.

When $n\equiv 4 \pmod{5}$, $\binom{n-1}{1}\equiv 3\pmod{5}$ and $\binom{n-2}{2}\equiv 1\pmod{5}$.

Since $\binom{n-1}{1}\not\equiv 0\pmod{5}$ and $\binom{n-2}{2}\not\equiv 0\pmod{5}$, we have $x^{d}$ and $x^{2d}$. Assume to the contrary that $\mathsf{E}_n$ is DO. Then $d=3^{\alpha}+1$ and $2d=3^i+1$, which implies $2(3^{\alpha}+1)=3^i+1$ if and only if $5^i-2\cdot 5^{\alpha}=1$, which is a contradiction. 

This completes the proof. 

\end{proof}

\begin{thm} \label{T2.5}
Let $p>5$. The polynomial $\mathsf{E}_n$ is a Dembowski-Ostrom polynomial over $\mathbb F_q$ if and only if the following holds.
\begin{enumerate}
\item [(i)] $n=2, 3$ and $d=p^k(p^{\alpha}+1)$. 
\end{enumerate}
\end{thm}

\begin{proof}

Let $n=2$. Then $\mathsf{E}_2=(p-1)\,x^d$. If $\mathsf{E}_2$ is DO, then clearly $d=p^{\alpha}+1$.

Let $n=3$. Then $\mathsf{E}_2=(p-2)\,x^d$. If $\mathsf{E}_3$ is DO, then clearly $d=p^{\alpha}+1$.

Now we show that when $n\not = 2, 3$, $\mathsf{E}_n$ is not DO. 

Assume that $n\not = 2, 3$. We consider two cases depending on the parity of $n$. 

\textbf{Case 1.} $n$ even.

\begin{equation*}
\begin{split}
&\mathsf{E}_n=-(n-1)x^d+\binom{n-2}{2}\,x^{2d}-\binom{n-3}{3}\,x^{3d}+\cdots +(-1)^{\frac{n}{2}}\,x^{\frac{dn}{2}}.
\end{split}
\end{equation*}

\textbf{Sub Case 1.1} $n\equiv 1 \pmod{p}$.

Since $\binom{n-2}{2}\equiv 1\pmod{p}$, we have $x^{2d}$ and $x^{\frac{dn}{2}}$. Assume to the contrary that $\mathsf{E}_n$ is DO. Then $2d=p^{\alpha}+1$ and $\frac{dn}{2}=p^i+1$, which implies $n(p^\alpha+1)=4p^i+4$. This is only true when $p=7$, $i=0$ and $\alpha >0$. Let $p=7$. Then clearly 
$\binom{n-3}{3}\not\equiv 3\pmod{7}$. Since $\mathsf{E}_n$ is DO, we have $3d=7^\beta +1$. Recall that $2d=7^{\alpha}+1$. Thus we have $2\cdot 7^\beta - 3\cdot 7^\alpha =1$, which is a contradiction.

\textbf{Sub Case 1.2} $n\not \equiv 1, 2, 3 \pmod{p}$.

We have $x^{d}$ and $x^{2d}$. Assume to the contrary that $\mathsf{E}_n$ is DO. Then $d=p^{\alpha}+1$ and $2d=p^i+1$, which implies $2(p^\alpha+1)=p^i+1$ if and only if $p^i-2p^\alpha =1$. This is only true when $i=0$ and $\alpha >0$, i.e. $2d=2$, which implies $d=1$. A contradiction since $d=p^{\alpha}+1$ and $\alpha >0$. 

\textbf{Sub Case 1.3}  $n\not \equiv 1 \pmod{p}$ and $n\equiv 2 \pmod{p}$.

Since $\binom{n-3}{3}\equiv -1\pmod{p}$, we have $x^{d}$ and $x^{3d}$. Assume to the contrary that $\mathsf{E}_n$ is DO. Then $d=p^{\alpha}+1$ and $3d=p^i+1$, which implies $3(p^\alpha+1)=p^i+1$ if and only if $p^i-3p^\alpha =2$. This is a contradiction for any $p>5$, $\alpha$ and $i$. 

\textbf{Sub Case 1.4}  $n\not \equiv 1 \pmod{p}$ and $n\equiv 3 \pmod{p}$.

Since $\binom{n-5}{5}\equiv -6\pmod{p}$, we have $x^{d}$ and $x^{5d}$. Assume to the contrary that $\mathsf{E}_n$ is DO. Then $d=p^{\alpha}+1$ and $5d=p^i+1$, which implies $5(p^\alpha+1)=p^i+1$ if and only if $p^i-5p^\alpha =4$. This is a contradiction for any $p>5$, $\alpha$ and $i$. 

\textbf{Case 2.} $n$ odd.

\begin{equation*}
\begin{split}
&\mathsf{E}_n=-(n-1)x^d+\binom{n-2}{2}\,x^{2d}-\binom{n-3}{3}\,x^{3d}+\cdots +(-1)^{\frac{n-1}{2}}\,\,\binom{\frac{n+1}{2}}{\frac{n-1}{2}}\,\,x^{\frac{d(n-1)}{2}}.
\end{split}
\end{equation*}

\textbf{Sub Case 2.1} $n\equiv 1 \pmod{p}$.

Since $\binom{n-2}{2}\equiv 1\pmod{p}$ and $\binom{n-3}{3}\equiv -4\pmod{p}$, we have $x^{2d}$ and $x^{3d}$. Assume to the contrary that $\mathsf{E}_n$ is DO. Then $2d=p^{\alpha}+1$ and $3d=p^i+1$, which implies $2\cdot p^i-3\cdot p^\alpha =1$. This is a contradiction for any $p>5$, $\alpha$ and $i$. 

\textbf{Sub Case 2.2}  $n\not \equiv 1, 2, 3 \pmod{p}$.

We have $x^{d}$ and $x^{2d}$. Assume to the contrary that $\mathsf{E}_n$ is DO. Then $d=p^{\alpha}+1$ and $2d=p^i+1$, which implies $2(p^\alpha+1)=p^i+1$ if and only if $p^i-2p^\alpha =1$. This is only true when $i=0$ and $\alpha >0$, i.e. $2d=2$, which implies $d=1$. A contradiction since $d=p^{\alpha}+1$ and $\alpha >0$. 

\textbf{Sub Case 2.3}  $n\not \equiv 1 \pmod{p}$ and $n\equiv 2 \pmod{p}$.

Since $\binom{n-3}{3}\equiv -1\pmod{p}$, we have $x^{d}$ and $x^{3d}$. Assume to the contrary that $\mathsf{E}_n$ is DO. Then $d=p^{\alpha}+1$ and $3d=p^i+1$, which implies $3(p^\alpha+1)=p^i+1$ if and only if $p^i-3p^\alpha =2$. This is a contradiction for any $p>5$, $\alpha$ and $i$. 

\textbf{Sub Case 2.4}  $n\not \equiv 1 \pmod{p}$ and $n\equiv 3 \pmod{p}$.

Since $\binom{n-5}{5}\equiv -6\pmod{p}$, we have $x^{d}$ and $x^{5d}$. Assume to the contrary that $\mathsf{E}_n$ is DO. Then $d=p^{\alpha}+1$ and $5d=p^i+1$, which implies $5(p^\alpha+1)=p^i+1$ if and only if $p^i-5p^\alpha =4$. This is a contradiction for any $p>5$, $\alpha$ and $i$. 

This completes the proof. 

\end{proof}

\begin{rmk}\label{R2}
The DO polynomials of the form $\mathsf{E}_n$, obtained in Theorem~\ref{T2.3}, Theorem~\ref{T2.4} and Theorem~\ref{T2.5}, are monomials, binomials, trinomials, and quadrinomials. We list them below. 
\begin{enumerate}
\item $p=3$, $d=p^k(p^\ell+1)$, $n=2$, $\textsc{E}_n=2x^{p^k(p^\ell+1)}$. 
\item $p=3$, $d=p^k(p^\ell+1)$, $n=3$, $\mathsf{E}_n=x^{p^k(p^\ell+1)}$. 
\item $p=3$, $d=p^k\Big(\frac{p^\ell+1}{2}\Big)$, $n=4$, $\mathsf{E}_n=x^{p^k(p^\ell+1)}$. 
\item $p=3$, $d=p^k(p^\ell+1)$, $n=5$, $\mathsf{E}_n=2x^{p^k(p^\ell+1)}$. 
\item $p=3$, $d=p^k(p^\ell+1)$, $n=6$, $\mathsf{E}_n=x^{p^k(p^\ell+1)}+2x^{p^{k+1}(p^\ell+1)}$. 
\item  $p=3$, $d=2p^k$, $n=7$, $\mathsf{E}_n=x^{4p^k}+2x^{2p^{k+1}}$. 
\item $p=3$, $d=2p^k$, $n=10$, $\mathsf{E}_n=x^{4p^k}+x^{2p^{k+1}}+2x^{10p^k}$. 
\item $p=3$, $d=2p^k$, $n=13$, $\mathsf{E}_n=x^{4p^k}+x^{10p^{k}}+x^{12p^k}$. 
\item $p=3$, $d=4p^k$, $n=15$, $\mathsf{E}_n=x^{4p^k}+x^{4p^{k+1}}+x^{28p^k}$. 
\item $p=3$, $d=2p^k$, $n=19$, $\mathsf{E}_n=x^{4p^k}+x^{2p^{k+1}}+2x^{10p^k}+2x^{2p^{k+2}}$.  
\item $p=5$, $d=p^k(p^\ell+1)$, $n=2$, $\mathsf{E}_n=4x^{p^k(p^\ell+1)}$. 
\item $p=5$, $d=p^k(p^\ell+1)$, $n=3$, $\mathsf{E}_n=3x^{p^k(p^\ell+1)}$. 
\item $p=5$, $d=2p^k$, $n=7$, $\mathsf{E}_n=4x^{2p^k}+x^{6p^k}$. 
\item $p>5$, $d=p^k(p^\ell+1)$, $n=2$, $\mathsf{E}_n=(p-1)\,x^{p^k(p^\ell+1)}$. 
\item $p>5$, $d=p^k(p^\ell+1)$, $n=3$, $\mathsf{E}_n=(p-2)\,x^{p^k(p^\ell+1)}$.
\end{enumerate}
\end{rmk}

%%%%%%%%%%%%%%%%%%%%%%%%%%%%%%%%%%%%%%

\end{document}